\documentclass[]{interact}

\usepackage{epstopdf}

\usepackage{amsmath, amssymb, amsfonts, bm, upgreek, mathrsfs, amscd}
\usepackage{amsthm}                

\usepackage{algorithm}             
\usepackage{algcompatible}         
\usepackage{algorithmicx}          

\usepackage{graphicx}              
\usepackage{epstopdf}              
\usepackage{multirow}              
\usepackage{booktabs}              
\usepackage{array}                 
\usepackage{float}                 

\usepackage{enumitem}              
\allowdisplaybreaks
\usepackage{caption}
\usepackage{subcaption}

\usepackage{xcolor}                
\usepackage{hyperref}              
\usepackage{textcomp}              

\usepackage{appendix}              
\usepackage{manyfoot}              
\usepackage{listings}              
\usepackage{etoolbox}              
\usepackage{lipsum}                

\usepackage{easyReview}            

\usepackage[numbers,sort&compress]{natbib}
\bibpunct[, ]{[}{]}{,}{n}{,}{,}
\renewcommand\bibfont{\fontsize{10}{12}\selectfont}

\newcommand{\Idx}[1]{\mathfrak{I}_{#1}}          
\newcommand{\Rplus}[1]{\mathbb{R}_{\geq}^{#1}}  
\newcommand{\xlm}{y^{\dagger}}                  
\newcommand{\xnlm}{y^{\dagger\dagger}}          

\newcommand{\obj}[1]{\mathcal{\phi}_{#1}}              

\newcommand{\fvec}{\boldsymbol{\bm{\Phi}}} 
\newcommand{\Feas}{\Omega}                      
\newcommand{\mult}{\mathsf{\eta}}                 
\newcommand{\muc}{\upmu}                       

\newcommand{\ff}{\mathcal{F}}
\newcommand{\basin}{\mathrm{U}}

\theoremstyle{plain}
\newtheorem{theorem}{Theorem}[section]

\theoremstyle{definition}
\newtheorem{definition}[theorem]{Definition}

\theoremstyle{remark}
\newtheorem{remark}{Remark}[section]%

\newtheorem{assumption}{Assumption}

\begin{document}


\title{One-parameter Filled Function Method for Non-convex Multi-objective Optimization Problems}

\author{
\name{Bikram Adhikary\textsuperscript{a} and Md Abu Talhamainuddin Ansary\textsuperscript{a}\thanks{CONTACT\\ Bikram Adhikary. Email:bikram.adhikary333@gmail.com\\ Md Abu Talhamainuddin Ansary. Email: md.abutalha2009@gmail.com (corresponding author)}}
\affil{\textsuperscript{a}Department of Mathematics, Indian Institute of Technology Jodhpur, Jodhpur, 342030, Rajasthan, India
}
}

\maketitle

\begin{abstract}
In this paper, a new one-parameter filled function approach is developed for nonlinear multi-objective optimization. Inspired by key filled function ideas from single-objective optimization, the proposed method is adapted to the multi-objective setting. It avoids scalarization weights and does not impose any prior preference ordering among objectives. A descent-based procedure is first applied to obtain a local weak efficient solution. An associated filled function is then constructed and used to derive another local weak efficient solution that improves upon the current one. Repeating these phases drives the search toward global weak efficiency and yields an approximation of the global Pareto front, including in non-convex problems where multiple local Pareto fronts may exist. Numerical experiments on a set of test problems demonstrate the effectiveness of the proposed approach relative to existing methods.
\end{abstract}




\begin{amscode}
\end{amscode}

\section{Introduction}\label{sec_intro}

Multi-objective optimization investigates problems with several objectives that are typically in conflict. A major difficulty is to compute Pareto-optimal solutions, for which no objective can be improved without sacrificing at least one other. Such formulations arise widely in science and engineering, for example in environmental analysis, space exploration, portfolio selection, management science, and healthcare. Since these trade-offs cannot be adequately captured by single-objective models, methods that balance competing objectives in a systematic and efficient manner are necessary.

\par 
{Conventional methods for multi-objective optimization are typically classified as scalarization methods or heuristic approaches. Scalarization techniques \cite{ Deb2001-tt, Ehrgott2005, Miettinen2012-rv} transform a multi-objective problem into a single-objective formulation via weighted combinations or constraint-based reformulations. However, performance depends strongly on parameter selection, and the full Pareto front is often not captured, particularly in non-convex settings. Heuristic approaches, including evolutionary algorithms \cite{deb2002fast,laumanns2002combining, mostaghim2007multi}, rely on stochastic search mechanisms but provide no theoretical guarantees of convergence to solutions. As a result, substantial computational cost and limited efficiency are often encountered in large-scale applications.}
\par
Recent developments in descent-based optimization have extended classical single-objective descent methods to multi-objective problems. Existing contributions address smooth unconstrained problems \cite{fliege2000steepest, ansary2015modified, ram2, gonccalves2022globally, mita2019nonmonotone, prudente2022quasi, peng2024novel, yahaya2025new, mahdavi2020superlinearly}, constrained scenarios \cite{ansary2020sequential, ansary2021sqcqp, fliege2016sqp, eichfelder2021general, custodio2018multiglods, jones2024constrained}, non-smooth optimization \cite{ansary2023proximal, bento2014proximal, tanabe2019proximal} and uncertain multi-objective optimization problems with finite uncertainty \cite{kumar2025steepest,kumar2,kumar2025modified} etc.. Although such methods are effective in many settings, convergence commonly occurs to local Pareto-optimal solutions, and the ability to address non-convexity remains limited. These limitations motivate the development of globally oriented descent frameworks. In particular, the method proposed in this paper seeks to improve exploration and approximation of the global Pareto front for non-convex multi-objective problems.

\par

Non-convex multi-objective problems are particularly difficult because the search landscape is typically irregular and the objectives compete in nontrivial ways. As a result, many distinct local Pareto-optimal solutions may exist, making the global Pareto front difficult to attain. Premature convergence to local solutions is frequently observed in gradient-based methods. By contrast, metaheuristics (e.g., genetic algorithms and particle swarm methods) can be computationally expensive and usually provide only limited theoretical guarantees. These observations highlight the need for an approach that can explore non-convex Pareto landscapes efficiently while remaining grounded in solid theoretical foundations.

\par In single-objective optimization, Renpu’s filled function method \cite{renpu1990filled} modifies the objective function so an algorithm can move away from non-global minima. Early versions were hard to tune and sometimes unstable. Later works, especially the variants \cite{zhang2004new,duanli2010global} improved stability and clarified convergence. Even so, many designs still use two interacting parameters that are tricky to calibrate and can hurt performance across problem classes. To ease that burden, Ge and Qin \cite{ge1987class} proposed a one-parameter filled function, and follow-up studies like \cite{he2019new,sui2019new,qu2022filled,yan2024novel} refined both analysis and practice. These one-parameter forms keep the useful theoretical properties and, under standard assumptions, provide a descent direction near superior points. This makes them practical for difficult non-convex landscapes and sets up our multiobjective extension.

\par {The two-parameter auxiliary function technique \cite{adhikary2025global} extends global descent method from single-objective optimization to the multi-objective situation, although it requires careful parameter tuning. To reduce this burden, we propose a one-parameter filled function method for multi-objective problems. Our aim is to explore non-convex and disconnected Pareto fronts in a systematic way, building on the underlying one-parameter theory. Multi-objective optimization brings additional obstacles, such as objective functions are not totally ordered, iterates may get stuck at locally Pareto-optimal points, gradients can conflict, and efficient sets may be disconnected. To address these kind of issues, we construct a multi-objective filled function with one-parameter that produces descent directions for all objective components at the same time. This helps the method to leave locally non-dominated regions and move toward globally Pareto-optimal solutions. As a result, this method supports more reliable exploration of difficult non-convex multi-objective problems, while retaining descent properties justified by theoretical foundations.}

\par The primary contributions of this research are as follows:
\begin{itemize}
  \item introduce a multi-objective auxiliary function, referred to as a \emph{filled function} (Definition~\ref{def_moff}), to enhance exploration in complex non-convex settings;
  \item design a \emph{one-parameter} multi-objective filled function that reduces tuning effort while preserving exploration capability;
  \item establish theoretical guarantees by proving the filled function conditions (F1–F3), identifying an admissible parameter interval, and showing the existence of descent directions near superior points together with convergence toward Pareto-critical solutions (Theorems~\ref{theorem_f1}, \ref{theorem_f2}, \ref{theorem_f3});
  \item formulate a filled function based algorithm that addresses limitations of classical scalarization and heuristic approaches, including principled restart/escape mechanisms from local Pareto traps;
  \item specify practical components, including step-size/line-search, parameter scheduling, to ensure robustness and low function evaluation cost; validate the method on benchmark problems using performance profiles; and compare against NSGA-II and recent deterministic multiobjective methods.
\end{itemize}

With this foundation, the proposed method systematically integrates a new multi-objective auxiliary function, termed the multi-objective filled function, into a robust multi-objective optimization framework. Traditional scalarization methods depend on predefined weights, whereas metaheuristic techniques often lack reliable convergence criteria. In contrast, the approach dynamically refines solutions to attain global efficiency. For non-convex problems, the method is well-suited and facilitates the finding of well-approximated Pareto fronts that conventional methods often fail to obtain. A new mechanism based on the multi-objective filled function algorithm is presented, enhancing multi-objective optimization techniques and expanding applicability to intricate optimization problems across several fields. Moreover, the resulting framework remains theoretically sound and delivers scalable, computationally efficient performance in high-dimensional multi-objective problems.


\par The paper is organized as follows. Section~\ref{sec_pre} first sets up the notation and summarizes the preliminary results required for the subsequent analysis. Section~\ref{OMFFM} then develops the multi-objective filled function method, introducing a new one-parameter multi-objective auxiliary function and detailing the resulting algorithm for global Pareto-front approximation. Next, Section~\ref{sec_imp_exp} examines the computational performance of the method through extensive numerical experiments. Finally, concluding remarks are presented, followed by a short discussion of possible directions for future research.


\section{Preliminaries}\label{sec_pre}
For any $p\in\mathbb{N}$, we employ the following notations throughout the paper:
\begin{align*}
\Idx{p} &:= \{1,2,\dots,p\};\\
\Rplus{p} &:= \bigl\{y\in\mathbb{R}^p:\ y_i\ge 0\ \text{for all } i\in\Idx{p}\bigr\}.
\end{align*}
Unless stated otherwise, vector inequalities in $\mathbb{R}^p$ are interpreted component wise.

Let us consider the multi-objective optimization problem:
\begin{equation*}
(\mathrm{MOP})\qquad \min_{y\in\Feas}\ \fvec(y)
:=\bigl(\obj{1}(y),\obj{2}(y),\dots,\obj{m}(y)\bigr),
\end{equation*}
where $m\ge 2$, each objective $\obj{j}:\mathbb{R}^n\to\mathbb{R}$ $(j\in\Idx{m})$ is continuously differentiable,
and the feasible set $\Feas$ which is a subset of $\mathbb{R}^n$ is compact, connected, and has nonempty interior.

\par
If a point \( y \in \Feas \) exists where all objective functions simultaneously achieve their minimum values, then \( y \) is called an ideal solution. Such a point is exceptional. In general, the objectives conflict, so improving one component of $\fvec$ may worsen another.
Consequently, multi-objective optimization is guided by efficiency (Pareto optimality) rather than a single optimal solution.

\par A point $\xlm\in\Feas$ is a \emph{global efficient solution} of $(\mathrm{MOP})$ if there is no $y\in\Feas$ satisfying
$\fvec(y)\le \fvec(\xlm)$ and $\fvec(y)\neq\fvec(\xlm)$.
It is a \emph{global weak efficient solution} if there is no $y\in\Feas$ such that $\fvec(y)<\fvec(\xlm)$.

\par Local notions are defined similarly.
A point $\xlm\in\Feas$ is a \emph{local efficient solution} if there exist $r>0$ and a neighborhood $N(\xlm,r)$ such that
no $y\in N(\xlm,r)\cap\Feas$ satisfies $\fvec(y)\le \fvec(\xlm)$ with $\fvec(y)\neq\fvec(\xlm)$.
It is a \emph{local weak efficient solution} if there exists $r>0$ such that no
$y\in N(\xlm,r)\cap\Feas$ satisfies $\fvec(y)<\fvec(\xlm)$.

\par If each $\obj{j}$ is convex, then any local (weak) efficient solution is also global.
Let $X^\dagger$ denote the set of all global/local (weak) efficient solutions of $(\mathrm{MOP})$.
The image $\fvec(X^\dagger)$ is called the corresponding global/local (weak) Pareto front.
Convex MOPs possess a single global Pareto front, whereas non-convex problems may exhibit multiple local Pareto fronts.

\par The next theorem provides a first-order necessary condition for weak efficiency, which follows from Lemma 3.1 of \cite{ansary2015modified}.

\begin{theorem}[{\bf First-order necessary condition for weak efficiency}]\label{theorem1.1}
If \linebreak $\xlm\in\mathrm{int}(\Feas)$ is a weak efficient solution of $(\mathrm{MOP})$, then there exists
$\mult\in\Rplus{m}$ with $\mult\neq \mathbf{0}^{m}$ such that
\begin{equation}\label{eq1.2}
\sum_{k\in\Idx{m}} \mult_k\,\nabla \obj{k}(\xlm)=0.
\end{equation}
\end{theorem}

Any $\xlm\in\mathrm{int}(\Feas)$ satisfying \eqref{eq1.2} is called a \emph{critical point} of $(\mathrm{MOP})$.
If all $\obj{j}$ are convex, then every critical point is weak efficient.
If all $\obj{j}$ are strictly convex, then every critical point is efficient.



\section{A new one-parameter multi-objective filled function method (OMFFM)}\label{OMFFM}
In this section, we begin by stating the assumptions used in the proposed framework. Next, we define a multi-objective filled function and then propose a one-parameter family for \(\mathrm{(MOPs)}\) based on this definition. We then verify that the required filled function conditions hold for this family.
These assumptions are used throughout the rest of the paper.
\begin{assumption}\label{assumption_lpf}
All local weak efficient solutions of (MOP) form a finite union of disjoint
local weak Pareto fronts $\{\mathcal{P}_1,\ldots,\mathcal{P}_K\}\subset \Feas$.
Accordingly, the set of all local weak efficient
solutions can be written as
\begin{align*}
\fvec^\dagger_{\text{all}}
&= \{\,\fvec(\xlm): \xlm \text{ is a local weak efficient solution of (MOP) over } \Feas \,\} \\[4pt]
&= \bigcup_{k=1}^K \{\,\fvec(y): y \in \mathcal{P}_k \,\}.
\end{align*}
\end{assumption}

\begin{assumption}\label{assumption4}
All globally efficient solutions of $(\mathrm{MOP})$ are contained in $\operatorname{int}(\Feas)$.

\end{assumption}
We now define the class of multi-objective filled functions. Let, \( \xlm \in \operatorname{int}(\Feas) \) be a known local weak efficient solution of $(\mathrm{MOP})$.
\begin{definition}[Multi-objective filled function]\label{def_moff}
    A function \( \ff_{\xlm} : \Feas \to \mathbb{R}^m \) is said to be a multi-objective filled function of $\fvec$ at \( \xlm \) if it satisfies the following conditions:
\begin{enumerate}[label=(F\arabic*)]
    \item \( \xlm \) is a local weak efficient solution of \( -\ff_{\xlm} \) over \( \Feas \);
    \item \( \ff_{\xlm} \) has no Fritz John points in the set:
    $$\basin_1= \left\{ y \in \mathrm{int}(\Feas) : y \neq \xlm,\ \obj{j}(y) \geq \obj{j}(\xlm) \text{ for at least one } j \in \Idx{m} \right\};$$

    \item Let $\mathcal{P}(\xlm)$ denote the local weak Pareto front containing $\xlm$. For any other front $\mathcal{P}_b \neq \mathcal{P}(\xlm)$ and any \linebreak
$\xnlm \in \mathcal{P}_b \cap \basin_2= \left\{ y \in \mathrm{int}(\Feas) :  \obj{j}(y) <\obj{j}(\xlm) \text{ for all } j \in \Idx{m} \right\}$, the function $\ff_{\xlm}$ admits a local weak efficient solution $y'$ such that $y' \in N(y^{\dagger\dagger},\epsilon) \subset \Feas$, and moreover $\obj{j}(y) < \obj{j}(\xlm) \quad \text{for all } y \in N(y^{\dagger\dagger},\epsilon),\ j \in \Idx{m}$.
\end{enumerate}
\end{definition}
The parameter \( \tau \) is, in principle, chosen to satisfy
\begin{equation}\label{eq2.2}
    0 < \tau < 
    \min \left\{\, \bigl|\obj{j}^{\dagger} - \obj{j}^{\dagger\dagger}\bigr| :
    \fvec^{\dagger} \in \mathcal{P}_a,\;
    \fvec^{\dagger\dagger} \in \mathcal{P}_b,\;
    \mathcal{P}_a \neq \mathcal{P}_b \right\},
\end{equation}
where each set \( \mathcal{P}_k \subset \fvec^\dagger_{\mathrm{all}} \) denotes a distinct local weak Pareto front. This choice guarantees that, whenever 
\( \fvec^{\dagger} \in \mathcal{P}_a \) and 
\( \fvec^{\dagger\dagger} \in \mathcal{P}_b \) with 
\( \mathcal{P}_a \neq \mathcal{P}_b \), 
the inequality \( \obj{j}^{\dagger\dagger} < \obj{j}^{\dagger} \) automatically implies $\obj{j}^{\dagger\dagger} < \obj{j}^{\dagger} - \tau$.


With this convention, our aim is to employ the resulting auxiliary function, referred to as the {multi-objective filled function}, to drive the algorithm so that, at each iteration, it moves from one local weak efficient solution of \(\mathrm{(MOP)}\) to another that is strictly better in the multi-objective sense.


\par To reduce the parameter-tuning burden of the two-parameter global descent method~\cite{adhikary2025global},
we propose a single-parameter filled function construction, inspired by~\cite{yan2024novel}, and establish its basic properties.
For $\nu>0$, define $\varphi_\nu:\mathbb{R}\to\mathbb{R}$ by
\begin{equation}\label{eq_sub}
\varphi_\nu(t)=
\begin{cases}
-\nu t^{3}, & t \ge 0,\\[2mm]
-\dfrac{1}{\nu} t^{2}, & t < 0.
\end{cases}
\end{equation}
It is immediate that $\varphi_\nu\in C^{1}(\mathbb{R})$.

Let $\xlm\in\Feas$ be a current local weak efficient solution of $\mathrm{(MOP)}$.
Define the one-parameter multi-objective filled function $\ff_{\xlm,\muc}:\Feas\to\mathbb{R}^{m}$ componentwise by
\begin{equation}\label{eq_filled}
\ff_{j,\xlm,\muc}(y)
=
-\|y-\xlm\|^{2}
+\varphi_{\muc}\!\big(\obj{j}(y)-\obj{j}(\xlm)\big),
\qquad j\in\Idx{m}.
\end{equation}
Here $\muc>0$ is the (single) tuning parameter. Since each $\obj{j}\in C^{1}(\mathbb{R}^{n})$,
the mapping $y\mapsto \|y-\xlm\|^{2}$ is smooth, and $\varphi_{\muc}\in C^{1}(\mathbb{R})$,
it follows that $\ff_{j,\xlm,\muc}\in C^{1}(\Feas;\mathbb{R})$ for all $j\in\Idx{m}$.



\begin{remark}
Under mild assumptions, the proposed family satisfies the multi-objective filled function conditions.
However, the verification of (F1)--(F3) is nontrivial. It requires a careful analysis and depends crucially on a suitable choice of the parameter $\muc$.
In particular, an admissible interval for $\muc$ is derived for which $\ff_{\xlm,\muc}$ satisfies (F1)--(F3).
\end{remark}

\begin{theorem}\label{theorem_f1}
 Let \( \xlm \in \mathrm{int}(\Feas) \) be a local weak efficient solution of  \( \mathrm{(MOP)} \). Then \( \xlm \) is a local weak efficient solution of \( -\ff_{\xlm,\muc} \) in $\Feas$. 
\end{theorem}
\begin{proof}
Assume \( \xlm \in \mathrm{int}(\Feas) \) is a local weak efficient solution of \( \mathrm{(MOP)} \). Then there exists an \( \epsilon \)-neighborhood \( N(\xlm,\epsilon) \subset \Feas \) such that no \( y \in N(\xlm,\epsilon) \) satisfies \linebreak \( \fvec(y) < \fvec(\xlm) \). This implies \( \obj{j}(y) \geq \obj{j}(\xlm) \) for at least one \( j \), for all \( y \in N(\xlm,\epsilon) \). This, together with \eqref{eq_sub} and \eqref{eq_filled}, yields:

\begin{align*}
\ff_{j,\xlm,\muc}(y) &= -\|y - \xlm\|^{2} -\muc(\obj{j}(y) - \obj{j}(\xlm))^3\\
&< 0 =\ff_{j,\xlm,\muc}(\xlm),\quad \text{for at least one } j, \quad \forall y \in N(\xlm,\epsilon) \setminus \{\xlm\}.
\end{align*}
This implies:

\begin{align*}
-\ff_{j,\xlm,\muc}(y) > -\ff_{j,\xlm,\muc}(\xlm),\quad \text{for at least one } j, \quad \forall y \in N(\xlm,\epsilon) \setminus \{\xlm\}.
\end{align*}
Thus, \( \xlm \) is a local weak efficient solution of \( -\ff_{\xlm,\muc} \) in $\Feas$.

\end{proof}

\begin{remark}
If \( \xlm \) is a global weak efficient solution of \( (MOP) \), then for every \( y\in\Feas\setminus\{\xlm\} \)
there exists \( j\in\Idx{m} \) such that \( \obj{j}(y)\ge \obj{j}(\xlm) \). For this index, \linebreak
\( \ff_{j,\xlm,\muc}(y) < 0 = \ff_{j,\xlm,\muc}(\xlm) \). Hence the same conclusion holds globally.
\end{remark}


\begin{theorem}\label{theorem_f2}
    Suppose $\xlm \in \Feas$ is a local weak efficient solution of $\mathrm{(MOP)}$. Then, for an appropriate value of $\muc$, the function \( \ff_{\xlm,\muc} \) admits no critical point in the set \( \basin_1 \).
\end{theorem}

\begin{proof}
Let $\bar{y} \in \basin_1 \subset \operatorname{int}(\Feas)$ be any arbitrary point. We show that, for a suitable choice of $\muc$, there exists a direction 
$s$ such that
\[
s^\top \nabla \ff_{j,\xlm,\muc}(\bar{y}) < 0,~\forall j \in \Idx{m}
\]
and hence $\bar{y}$ can not be a critical point of $\ff_{\xlm,\muc}$.




The proof proceeds by constructing three partitions of the index set $\Idx{m}$. For each partition, a direction $s$ is selected such that
$s^\top(\bar{y}-\xlm)>0$. A natural choice is $s=\bar{y}-\xlm$. It is then established that, for suitable values of $\muc$, this vector $s$ serves as a descent direction for the filled function.


We distinguish the following three cases, depending on the partitions of $\Idx{m}$. Let \( \mathcal{I}_1(\bar{y}), \mathcal{I}_2(\bar{y}), \mathcal{I}_3(\bar{y}) \) be a partitions of
\( \Idx{m} \) such that
\begin{enumerate}
    \item for all \( j \in \mathcal{I}_1(\bar{y}) \), \( \obj{j}(\bar{y}) = \obj{j}(\xlm) \);
    \item for all \( j \in \mathcal{I}_2(\bar{y}) \), \( \obj{j}(\bar{y}) > \obj{j}(\xlm) \);
    \item for all \( j \in \mathcal{I}_3(\bar{y})\), \( \obj{j}(\bar{y}) < \obj{j}(\xlm) \).
\end{enumerate}

\medskip
\noindent\textbf{Case I:}
\emph{Consider \( j \in \mathcal{I}_1(\bar{y}) \).}

    For all \( j \in \mathcal{I}_1(\bar{y}) \), \( \obj{j}(\bar{y}) = \obj{j}(\xlm) \), and
    from \eqref{eq_sub} and \eqref{eq_filled} we have
    \[
    \ff_{j,\xlm,\muc}(y) = -\|y - \xlm\|^{2}, 
    \quad
    \nabla \ff_{j,\xlm,\muc}(\bar{y}) = -2(\bar{y} - \xlm).
    \]
    Thus, for all $j \in \mathcal{I}_1$,
    \[
    s^\top \nabla \ff_{i,\xlm,\muc}(\bar{y})
    = -2\,s^\top (\bar{y} - \xlm) < 0.
    \]

    \medskip
\noindent\textbf{Case II:} \emph{Consider \( j \in \mathcal{I}_2(\bar{y}) \).}
\par Let \( s \) be any direction such that \( s^\top (\bar{y} - \xlm) > 0 \).
Now for the partitions \( \mathcal{P}_1, \mathcal{P}_2, \mathcal{P}_3 \) of
\( \mathcal{I}_2(\bar{y}) \), assume that
\begin{enumerate}[label=II\alph*.]
    \item \( s^\top \nabla \obj{i}(\bar{y}) = 0 \), for all \( i \in \mathcal{P}_1 \);
    \item \( s^\top \nabla \obj{i}(\bar{y}) > 0 \), for all \( i \in \mathcal{P}_2 \);
    \item \( s^\top \nabla \obj{i}(\bar{y}) < 0 \), for all \( i \in \mathcal{P}_3 \);
\end{enumerate}

    For all \( j \in \mathcal{I}_2(\bar{y}) \), \( \obj{j}(\bar{y}) > \obj{j}(\xlm) \), and
    \[
    \ff_{j,\xlm,\muc}(y) 
    = -\|y - \xlm\|^{2} -\muc\big(\obj{j}(y) - \obj{j}(\xlm)\big)^3,
    \]
    so
    \[
    \nabla \ff_{j,\xlm,\muc}(\bar{y})
    =-2(\bar{y}-\xlm)
    -3\muc(\obj{j}(\bar{y}) - \obj{j}(\xlm))^2 \nabla \obj{j}(\bar{y}).
    \]
    Hence
    \[
    s^\top \nabla \ff_{j,\xlm,\muc}(\bar{y})
    =-2s^\top(\bar{y}-\xlm)
    -3\muc(\obj{j}(\bar{y}) - \obj{j}(\xlm))^2 s^\top \nabla \obj{j}(\bar{y}).
    \]

    Using the sign conditions for the sets $\mathcal{P}_1,\mathcal{P}_2,\mathcal{P}_3$:

    \begin{itemize}
        \item If \( i \in \mathcal{P}_1 \), then \( s^\top \nabla \obj{i}(\bar{y}) = 0 \), and for $0< \muc < 1$, we have
        \[
        s^\top \nabla \ff_{i,\xlm,\muc}(\bar{y})
        = -2s^\top(\bar{y}-\xlm) < 0.
        \]

        \item If \( i \in \mathcal{P}_2 \), then \( s^\top \nabla \obj{i}(\bar{y}) > 0 \), and the
        second term is also strictly negative. Hence for $0 < \muc < 1$, we have
        \[
        s^\top \nabla \ff_{i,\xlm,\muc}(\bar{y}) < 0.
        \]

        \item If \( i \in \mathcal{P}_3 \), then \( s^\top \nabla \obj{i}(\bar{y}) < 0 \).
        In this case, the second term is positive, and for
        \[
        0 < \muc < 
        \min \left\{ 1,\,
        \frac{2s^\top (\bar{y} - \xlm)}{\max_{i \in \mathcal{P}_3} \left( -3 (\obj{i}(\bar{y}) - \obj{i}(\xlm))^2 \cdot s^\top \nabla \obj{i}(\bar{y}) \right)} \right\},
        \]
        we obtain
        \[
        s^\top \nabla \ff_{i,\xlm,\muc}(\bar{y}) < 0.
        \]
    \end{itemize}



    \medskip
\noindent\textbf{Case III:} \emph{Consider \( j \in \mathcal{I}_3(\bar{y}) \).}  
\par Let \( s \) be any direction such that \( s^\top (\bar{y} - \xlm) > 0 \) holds.
For the partitions \( \mathcal{Q}_1, \mathcal{Q}_2, \mathcal{Q}_3 \) of
\( \mathcal{I}_3(\bar{y}) \), assume that
\begin{enumerate}[label=III\alph*.]
    \item \( s^\top \nabla \obj{i}(\bar{y}) = 0 \), for all \( i \in \mathcal{Q}_1 \);
    \item \( s^\top \nabla \obj{i}(\bar{y}) > 0 \), for all \( i \in \mathcal{Q}_2 \);
    \item \( s^\top \nabla \obj{i}(\bar{y}) < 0 \), for all \( i \in \mathcal{Q}_3 \);
\end{enumerate}

    For all \( j \in \mathcal{I}_3(\bar{y}) \), \( \obj{j}(\bar{y}) < \obj{j}(\xlm) \), and
    \[
    \ff_{j,\muc,\xlm}(y) 
    = -\|y - \xlm\|^{2} -\frac{1}{\muc}\big(\obj{j}(y) - \obj{j}(\xlm)\big)^2,
    \]
    so
    \[
    \nabla \ff_{j,\muc,\xlm}(\bar{y})
    =-2(\bar{y}-\xlm)
    -\frac{2}{\muc}(\obj{j}(\bar{y}) - \obj{j}(\xlm)) \nabla \obj{j}(\bar{y}),
    \]
    and
    \[
    s^\top \nabla \ff_{j,\muc,\xlm}(\bar{y})
    =-2s^\top(\bar{y}-\xlm)
    -\frac{2}{\muc}(\obj{j}(\bar{y}) - \obj{j}(\xlm)) s^\top \nabla \obj{j}(\bar{y}).
    \]
    Since here \(\obj{j}(\bar{y}) - \obj{j}(\xlm) < 0\):

    \begin{itemize}
        \item If \( i \in \mathcal{Q}_1 \), then \( s^\top \nabla \obj{i}(\bar{y}) = 0 \), and for $0 < \muc < 1$,
        \[
        s^\top \nabla \ff_{i,\muc,\xlm}(\bar{y})
        = -2s^\top(\bar{y}-\xlm) < 0.
        \]

        \item If \( i \in \mathcal{Q}_2 \), then \( s^\top \nabla \obj{i}(\bar{y}) > 0 \), and the
        second term is positive. Thus, for
        \[
        \muc >   
        \frac{\max_{i \in \mathcal{Q}_2} \left( -(\obj{i}(\bar{y}) - \obj{i}(\xlm)) \cdot s^\top \nabla \obj{i}(\bar{y}) \right)}{s^\top (\bar{y} - \xlm)},
        \]
        we obtain
        \[
        s^\top \nabla \ff_{i,\muc,\xlm}(\bar{y}) < 0.
        \]

        \item If \( i \in \mathcal{Q}_3 \), then \( s^\top \nabla \obj{i}(\bar{y}) < 0 \), so
        the second term is strictly negative. Hence for \linebreak
        $0 < \muc < 1$, we have
        \[
        s^\top \nabla \ff_{i,\xlm,\muc}(\bar{y}) < 0.
        \]
    \end{itemize}

By examining all cases induced by the partition of \(\Idx{m}\), we deduce that for any direction \(s\) such that
\[
s^\top(\bar{y}-\xlm)>0,
\]
and for any \(\muc\in(\muc_L,\muc_U)\), where
\[
\muc_L=
\max\left\{
0,\,
\frac{\max_{i\in\mathcal{Q}_2}\left(-(\obj{i}(\bar{y})-\obj{i}(\xlm))\, \cdot s^\top \nabla \obj{i}(\bar{y})\right)}
{s^\top(\bar{y}-\xlm)}
\right\},
\]
\[
\muc_U=
\min\left\{
1,\,
\frac{2\,s^\top(\bar{y}-\xlm)}
{\max_{i\in\mathcal{P}_3}\left(-3(\obj{i}(\bar{y})-\obj{i}(\xlm))^2\, \cdot s^\top \nabla \obj{i}(\bar{y})\right)}
\right\},
\]
the direction \(s\) is a descent direction of \(\ff_{\xlm,\muc}\) at \(\bar{y}\). Therefore, \(\bar{y}\) cannot be a critical point of \(\ff_{\xlm,\muc}\). Since \(\bar{y}\in\basin_1\) was arbitrary, we conclude that \(\ff_{\xlm,\muc}\) has no critical point in \(\basin_1\). This completes the proof.

\end{proof}

\begin{remark}
If $\mathcal{Q}_2\neq \emptyset$ and $\muc \leq \frac{\max_{i \in \mathcal{Q}_2}\Bigl(-\bigl(\obj{i}(\bar{y})-\obj{i}(\xlm)\bigr)\, \cdot s^\top \nabla \obj{i}(\bar{y})\Bigr)}
{s^\top (\bar{y}-\xlm)}$, then by the preceding theorem, the multi-objective filled function $\ff_{\xlm,\muc}$ admits no descent
direction in $\basin_1$. Consequently, the method cannot generate an improving iterate
within $\basin_1$, and hence $\xlm$ is declared as a global weak efficient solution of $\mathrm{(MOP)}$.

On the other hand, if $\mathcal{Q}_2=\emptyset$ and $\muc \rightarrow 0$, then the standing
assumptions required for $\varphi_{\muc}$ are violated and in this case as well, the preceding theorem
implies that $\xlm$ is declared as a global weak efficient solution of $\mathrm{(MOP)}$.
\end{remark}



\begin{theorem}\label{theorem_f3}
Let $\xnlm \in \operatorname{int}(\Feas)$ be a local weak efficient 
solution of the problem $\mathrm{(MOP)}$, that lies on a local weak Pareto front different from that of $\xlm$, with \linebreak $\obj{j}(y^{\dagger\dagger}) < \obj{j}(\xlm)$ for all $j \in \Idx{m}$. If \(\muc>0\) is sufficiently small, then \(\ff_{\xlm,\muc}\) has a weak efficient solution \(y'\in\Feas\) such that
\[
y' \in N(y^{\dagger\dagger},\epsilon)\subset \Feas
\quad \text{and} \quad
\obj{j}(y') < \obj{j}(\xlm)-\tau,\quad j\in\Idx{m},
\]
where \(N(y^{\dagger\dagger},\epsilon)\) satisfies \(\obj{j}(y)<\obj{j}(\xlm)\) for all
\(y\in N(y^{\dagger\dagger},\epsilon)\) and all \(j\in\Idx{m}\).
\end{theorem}

\begin{proof}
Since \( \Feas \) is compact, \( \obj{j} \) are continuous over \( \Feas \), and \( \obj{j}(y^{\dagger\dagger}) < \obj{j}(\xlm) \), for sufficiently small \( \epsilon > 0 \), there exists a neighborhood \( N(y^{\dagger\dagger},\epsilon) \subset \Feas \) such that:
\[
\obj{j}(y^{\dagger\dagger}) \leq \obj{j}(y) < \obj{j}(\xlm), 
\quad \forall y \in N(y^{\dagger\dagger},\epsilon),\ j \in \Idx{m}.
\]
From (\ref{eq2.2}), $\obj{j}(y^{\dagger\dagger})  < \obj{j}(\xlm)$ implies:
\[
\obj{j}(y^{\dagger\dagger}) < \obj{j}(\xlm) - \tau,\qquad j \in \Idx{m}.
\]
If $\delta > 0$ is sufficiently small, then:
\begin{equation}\label{eq11}
    \obj{j}(y^{\dagger\dagger}) + \delta \leq \obj{j}(\xlm) - \tau,
    \quad j \in \Idx{m}.
\end{equation}
Define:
\begin{align}
 U^{\leq}_{\delta} = \{ y \in N(y^{\dagger\dagger},\epsilon) : \obj{j}(y) \leq \obj{j}(y^{\dagger\dagger}) + \delta,\ j \in \Idx{m}\}. \label{eq12}
\end{align}

where for all $y \in \mathrm{bd}( U^{\leq}_{\delta})$, it holds that $$\obj{j}(y)-\obj{j}(\xlm)=\obj{j}(\xnlm)+\delta-\obj{j}(\xlm)\leq -\tau<0.$$ Clearly,
\begin{align*}
  \obj{j}(\xnlm)-\obj{j}(\xlm)&<\obj{j}(\xnlm)+\delta-\obj{j}(\xlm),\quad \forall y\in \mathrm{bd}(U^{\leq}_{\delta}),\\
  &= \obj{j}(y)-\obj{j}(\xlm)<0,\quad \forall y\in \mathrm{bd}(U^{\leq}_{\delta}),
\end{align*}

which implies that:
\[
\big(\obj{j}(\xnlm)-\obj{j}(\xlm)\big)^{2}>\big(\obj{j}(y)-\obj{j}(\xlm)\big)^{2},\quad \forall y\in \mathrm{bd}(U^{\leq}_{\delta}).
\]
Meanwhile, for each $y \in \mathrm{bd}(U^{\leq}_{\delta})$, there are two cases:

\noindent (i) $\|\xnlm-\xlm\|\ge \|y-\xlm\|$;

\noindent (ii) $\|\xnlm-\xlm\|< \|y-\xlm\|$.

\noindent For case (i), we have
\begin{align*}
\ff_{j,\xlm,\muc}(\xnlm)-\ff_{j,\xlm,\muc}(y)
&=-\|\xnlm-\xlm\|^{2}-\frac{1}{\muc}\big(\obj{j}(\xnlm)-\obj{j}(\xlm)\big)^{2}
+\|y-\xlm\|^{2}+\frac{1}{\muc}\big(\obj{j}(y)-\obj{j}(\xlm)\big)^{2}\\
&=-\Big(\|\xnlm-\xlm\|^{2}-\|y-\xlm\|^{2}\Big)
-\frac{1}{\muc}\Big(\big(\obj{j}(\xnlm)-\obj{j}(\xlm)\big)^{2}-\big(\obj{j}(y)-\obj{j}(\xlm)\big)^{2}\Big)\\
&<0.
\end{align*}
For case (ii), when
\[
\mu<\frac{\min_{j \in \Idx{m}}\big[\big(\obj{j}(\xnlm)-\obj{j}(\xlm)\big)^{2}-\big(\obj{j}(y)-\obj{j}(\xlm)\big)^{2}\big]}
{\|y-\xlm\|^{2}-\|\xnlm-\xlm\|^{2}},
\]
it holds that:

$$-\Big(\|\xnlm-\xlm\|^{2}-\|y-\xlm\|^{2}\Big) < \frac{1}{\muc}\Big(\big(\obj{j}(\xnlm)-\obj{j}(\xlm)\big)^{2}-\big(\obj{j}(y)-\obj{j}(\xlm)\big)^{2}\Big),$$
which is equivalent to:
\[
-\|\xnlm-\xlm\|^{2}-\frac{1}{\mu}\big(\obj{j}(\xnlm)-\obj{j}(\xlm)\big)^{2}
<-\|y-\xlm\|^{2}-\frac{1}{\mu}\big(\obj{j}(y)-\obj{j}(\xlm)\big)^{2}.
\]
Hence,
\begin{equation} \label{eq13}
    \ff_{j,\xlm,\muc}(\xnlm) < \ff_{j,\xlm,\muc}(y), \quad  \forall y \in \mathrm{bd}(U^{\leq}_{\delta}),~ j \in \Idx{m}.
\end{equation}

 
By compactness of $U^{\leq}_{\delta}$ and continuity of $\ff_{\xlm, \muc}$, there exists at least one weakly Pareto-optimal point $y'$ of $\ff_{\xlm, \muc}$ over $U^{\leq}_{\delta}$. By \eqref{eq13}, no such point lies in $\mathrm{bd}(U^{\leq}_{\delta})$, so $y'\in \mathrm{int}(U^{\leq}_{\delta})$. Thus
\[
\obj{j}(y') < \obj{j}(y^{\dagger\dagger}) + \delta 
\leq \obj{j}(\xlm) - \tau,\qquad j \in \Idx{m}.
\]
Finally, since \( N(y^{\dagger\dagger},\epsilon) \)  is an open set and $ \mathrm{int}(U^{\leq}_{\delta})$ is bounded, then by the continuity of \( \ff_{j, \xlm,\muc} \), there exists \( 0 < \varepsilon_1 < \varepsilon \) such that whenever \( 0 < \varepsilon' \leq \varepsilon_1 \), we have:
\begin{equation*}
N(\varepsilon',y') \subset \mathrm{int}(U^{\leq}_{\delta}) \quad \text{and} \quad
\ff_{j, \xlm,\muc}(y') \leq \ff_{j, \xlm,\muc}(y), \quad \forall y \in N(\varepsilon',y'),~j \in \Idx{m},
\end{equation*}
where \( N(\varepsilon',y') \) is an \( \varepsilon' \)-neighborhood of \( y' \).\\
Therefore, \( y' \) is a weak efficient solution of $\ff_{\xlm,\muc}$ over \( \Feas \).
\end{proof}


\begin{remark}
Suppose there is no point \(y\in\Feas\) (in particular, no local weak Pareto front different from that of \(\xlm\))
such that \( \obj{j}(y) < \obj{j}(\xlm) \) for all \( j\in\Idx{m} \).
Then no feasible point strictly dominates \(\xlm\), and hence \(\xlm\) is a global weak efficient solution of \((MOP)\).
In this case, the situation covered by Theorem~\ref{theorem_f3} cannot occur, and the filled function iteration has no strictly improving front to move to.
\end{remark}

\subsection{Algorithm}
\label{sec_alg}

Building on the theoretical results established in the previous sections, an algorithm for the one-parameter multi-objective filled function method is developed in this subsection. The proposed approach targets non-convex multi-objective optimization problems and is intended to locate globally efficient solutions. In contrast to the single-objective setting, where the aim is a single minimizer, the objective here is to approximate an entire family of efficient (Pareto-optimal) solutions. The initial point generation technique outlined in~\cite{ansary2021sqcqp} (Steps~1--4 of Algorithm~6.1) is integrated into the scheme to provide a well-distributed numerical representation of this collection. These materials are then incorporated into the algorithm provided below. The Pareto front associated with the initial problem prior to the application of the multi-objective filled function approach is denoted by $(\mathrm{PF})$ throughout this subsection, whereas the equivalent Pareto front produced following its application is denoted by $(\mathrm{PFF})$. Similarly, the weak Pareto fronts prior to and following the application of the filled function approach are denoted by $(\mathrm{WPF})$ and $(\mathrm{WPFF})$ accordingly.

\begin{algorithm}[H]
\caption{({\it The core framework of OMFFM})} \label{alg1}
\vspace{0.5mm}
\begin{algorithmic}
\State \textbf{1. Initialization phase}
\begin{enumerate}[label=(\alph*)]
        \item Choose a nonempty subset $\Feas_0 \subset \Feas$ with N initial points.
        \item Steps~1--4 of algorithm~6.1 in~\cite{ansary2021sqcqp} are used to update \(N\) initial points that satisfy inequality~(5.1) in that work. The ideal and nadir vectors are then obtained by applying a single-objective global descent method \cite{duanli2010global} to each objective \( j \in \Idx{m} \).
        \item Choose initial parameters: \( 0 < \muc_{\text{ini}} < 1 \),
        \( 0 < \hat{\muc} < 1 \), $\muc_L > 0$, \( \epsilon > 0 \), \( \kappa > 0 \), and \( \bar{\beta}_U > 0 \).
        \item Set $WPF = WPFF = PF = PFF = \emptyset$. 
        \item Set \( \muc := \muc_{\text{ini}} \).
\end{enumerate}

\State \textbf{2. Local phase}
\begin{enumerate}[label=(\alph*)]
    \item From the current point \( y_{\text{cur}} \), apply an appropriate local optimization procedure to find a local weak efficient solution \( \xnlm \) of \( \mathrm{(MOP)} \).
    \item Update $\xlm:= \xnlm$ and $WPF = WPF \cup \{\fvec(\xlm)\}$.
    \item Go to the global phase.
\end{enumerate}

\State \textbf{3. Global phase}
\begin{enumerate}[label=(\alph*),,series=OMFFM]
    \item  Create \( m \) number of points as: $\{y_{\text{ini}}^{(k)} \in \Feas \setminus N(\xlm,\epsilon)\}_{k=1}^m$ and set $k := 1$. \label{step_2a}
    \item Set $y_{\text{cur}} := y_{\text{ini}}^{(k)}$. \label{step_2b}
    \item Update $\muc_L$.
    \item If \( \fvec(y_{\text{cur}}) < \fvec(\xlm) \), then return to the local phase. Else, go to step 3\ref{step_2d}.\label{step_2c}

   \end{enumerate}

    \algstore{myalg0}
\end{algorithmic}
\end{algorithm}

\begin{algorithm}[H]
\ContinuedFloat
\caption{({\it The core framework of OMFFM}) (continued)}
\vspace{0.5mm}
\begin{algorithmic}
\algrestore{myalg0}

\Statex 

\begin{enumerate}[label=(\alph*),resume=OMFFM]

    \item If, for some \(j\in\Idx{m}\),
          \(\|\nabla \ff_{j,\xlm,\muc}(y_{\text{cur}})\| < \kappa\) or
          \((y_{\text{cur}}-\xlm)^{\top}\nabla \ff_{j,\xlm,\muc}(y_{\text{cur}})\ge 0\),
          then select \(l\in\mathbb{N}\) such that \(\muc_l := \hat{\muc}^{\,l}\muc\) and, for all \(j\in\Idx{m}\),
          \(\|\nabla \ff_{j,\xlm,\muc_l}(y_{\text{cur}})\|\ge \kappa\) and
          \((y_{\text{cur}}-\xlm)^{\top}\nabla \ff_{j,\xlm,\muc_l}(y_{\text{cur}})<0\).\vspace{0.3mm}
 \newline\hspace*{1em}Replace \(\muc := \muc_l\). \label{step_2d}


      \item Compute a descent direction \(\bar{s}\) for \(\ff_{\xlm,\muc}\) at \(y_{\text{cur}}\).
          Choose \(\bar{\beta}\le \bar{\beta}_U\) by line search and set
          \(y_{\text{new}} := y_{\text{cur}}+\bar{\beta}\bar{s}\).
          Update \(y_{\text{cur}} := y_{\text{new}}\) and return to step~3\ref{step_2c}.
    \item If \( y_{\text{cur}} \) lies on the boundary of \( \Feas \), terminate the current step and move to step~3\ref{step_2h}.
    \item Set \( k := k + 1 \). If \( k \leq m \), return to step~3\ref{step_2b}. \label{step_2h}



    \item Otherwise, set $\muc := \hat{\muc}\,\muc$. If $\muc \ge \muc_L$, return to Step~3\ref{step_2a}. If $\muc < \muc_L$, terminate the algorithm, declare $\xlm$ as a global weak efficient solution obtained, and update $ WPFF = WPFF \cup \{\fvec(\xlm)\}$.
    \item Obtain \(PF\) by extracting the non-dominated solutions from \(WPF\).
    \item Obtain \(PFF\) by extracting the non-dominated solutions from \(WPFF\).
\end{enumerate}
\end{algorithmic}
\end{algorithm}

\section{Computational implementation and numerical results}\label{sec_imp_exp}
This section describes the MATLAB (R2024b) implementation of $\mathrm{OMFFM}$ (Algorithm~\ref{alg1}) and the numerical workflow used in the computational experiments, including the resulting performance profiles. As defined in the preliminaries, inequalities between objective functions are interpreted componentwise.

\subsection{Implementation of {OMFFM}}\label{subsec_impl_omffm}
The implementation is structured around the three phases of Algorithm~\ref{alg1}.
During the run, two archives are maintained:
$WPF$ records objective vectors produced by the local phase, while $WPFF$ stores objective vectors reported at termination of the filled function refinement.
The fronts $PF$ and $PFF$ are then obtained by extracting non-dominated vectors from $WPF$ and $WPFF$, respectively.

\begin{itemize}
\item \textit{Initialization and point refinement:}
First, $N$ points are sampled uniformly from $[\mathrm{lb},\mathrm{ub}]\subset\Feas\subset\mathbb{R}^n$. 
Next, Steps~1--4 of Algorithm~6.1 in~\cite{ansary2021sqcqp} are applied to refine these points so that inequality~(5.1) in that reference is satisfied. 
The refined points form the initial set $\Feas_0$ used in Algorithm~\ref{alg1}.

\item \textit{Ideal and nadir estimation:}
Ideal ($\fvec^I$) and nadir ($\fvec^N$) vectors are estimated via a payoff-table style procedure~\cite{Miettinen2012-rv}.
For each objective $\obj{j}$, a single-objective global descent method~\cite{duanli2010global} is executed from representative initial points (e.g., $\mathrm{lb}$, $(\mathrm{lb}+\mathrm{ub})/2$, and $\mathrm{ub}$).
Non-dominated outcomes from these runs are used to approximate $\fvec^I$ and $\fvec^N$.

\item \textit{Parameter setup:} For the filled function method, the parameters are set as follows:
$\muc_{\mathrm{ini}}=0.01$, $\hat{\muc}=0.005$, $\epsilon=0.1$ (with $\bar\beta_U=\epsilon>0$),
and $\kappa=10^{-4}\sqrt{n}$, where $n$ denotes the problem dimension. The parameter
$\muc$ is capped above by $\muc_U=1$. In Step~3\ref{step_2d}, we set $l=1$.


\item \textit{Local phase:}
Given a current point $y_{\mathrm{cur}}$, a local multi-objective solver is called to compute a local weak efficient point $\xnlm\in\Feas$.
After setting $\xlm:=\xnlm$, the objective vector $\fvec(\xlm)$ is appended to $WPF$.
In practice, this step can be realized using a penalty-based SQCQP-type routine (or an equivalent local method), typically implemented with \texttt{fmincon} as an inner engine.

\item \textit{Global phase:}
\begin{itemize}
    \item For global refinement, $m=2n$ trial points are constructed in $\Feas\setminus N(\xlm,\epsilon)$.
The candidates may be produced by rejection sampling over $\Feas$ or by symmetric perturbations around $\xlm$ followed by feasibility restoration (projection for box constraints, or constraint handling for general feasible regions).

\item 
The global phase returns to the local phase whenever a trial iterate yields a componentwise strict improvement,
$\fvec(y_{\mathrm{cur}})<\fvec(\xlm)$.
The condition is deliberately strong and activates only when all objectives improve simultaneously, consistent with Step~3\ref{step_2c} of Algorithm~\ref{alg1}.

\item During the global refinement phase, the lower bound $\muc_L$ for the parameter $\muc$ is updated adaptively using the theoretical bound.
Let $\bar y:=y_{\mathrm{cur}}\neq \xlm$ and set $s:=\bar y-\xlm$.
Assume that $\obj{j}(\bar y)<\obj{j}(\xlm)$ for all $j\in\mathcal{I}_3(\bar y)$.
If $\mathcal{Q}_2\neq\emptyset$, then
\[
\muc_L \;:=\; 
\frac{\max_{i \in \mathcal{Q}_2} \left( -(\obj{i}(\bar{y}) - \obj{i}(\xlm)) \, s^\top \nabla \obj{i}(\bar{y}) \right)}
{s^\top (\bar{y} - \xlm)}+ 10^{-5};
\]
otherwise, $\muc_L:=10^{-5}$.
If the resulting value is below the prescribed cap $\muc_U$, it is accepted; otherwise, $\muc_L$ is reset to $10^{-5}$.



\item 
If the descent checks in step~3\ref{step_2d} fail for at least one index $j\in\Idx{m}$, the parameter is decreased as $\muc_l=\hat\muc^l\muc$ for some $l\in\mathbb{N}$ until the descent conditions hold across all objectives.

\item 
A descent direction $\bar s$ is computed for the filled function model at $y_{\mathrm{cur}}$. A backtracking line search selects $\bar\beta\le \bar\beta_U$ and updates $y_{\mathrm{new}}=y_{\mathrm{cur}}+\bar\beta\,\bar s$, accepting the step once the descent condition is satisfied and feasibility is preserved.
If the update reaches boundary, refinement of the current start point is stopped and the method proceeds to the next candidate.

\item 
After all $m$ trial starts have been processed, an outer reduction $\muc:=\hat\muc\,\muc$ is performed.
If $\muc\ge \muc_L$, a new refinement round begins; otherwise, the algorithm terminates and records $\fvec(\xlm)$ in $WPFF$.

\end{itemize}

\item \textit{Front construction:}
Finally, non-dominated filtering is applied to $WPF$ and $WPFF$ to produce $PF$ and $PFF$, respectively.
\end{itemize}

\paragraph*{Numerical Experiments}
The performance of Algorithm~\ref{alg1} is assessed on a diverse suite of test problems, all considered as box-constrained.
The results indicate that the proposed method is robust and computationally efficient across the benchmark set, relative to established solvers. Table~\ref{table_1} summarizes the test problems and their sources.

\begin{table}[!htbp]
\centering
\scalebox{0.83}{
\scriptsize
\setlength{\tabcolsep}{6pt}
\renewcommand{\arraystretch}{1.2}
\begin{tabular}{|c|l|c|c|c|l|c|c|}
\hline
\textbf{Sl. no.} & \textbf{Test problem} & \textbf{(m,n)} & \textbf{Source} &
\textbf{Sl. no.} & \textbf{Test problem} & \textbf{(m,n)} & \textbf{Source} \\
\hline

1	&	AL1	&	(2,20)	&	\cite{adhikary2025global}	&	32	&	lovison5	&	(3,3)	&	\cite{fliege2016sqp} (Table 1)	\\
2	&	AL2	&	(2,50)	&	\cite{adhikary2025global}	&	33	&	lovison6	&	(3,3)	&	\cite{fliege2016sqp} (Table 1)	\\
3	&	CEC09\_1	&	(2,30)	&	\cite{fliege2016sqp} (Table 1)	&	34	&	LP1	&	(2,50)	&	\cite{adhikary2025global}	\\
4	&	CEC09\_2	&	(2,30)	&	\cite{fliege2016sqp} (Table 1)	&	35	&	LR1	&	(2,50)	&	\cite{adhikary2025global}	\\
5	&	CEC09\_3	&	(2,30)	&	\cite{fliege2016sqp} (Table 1)	&	36	&	MOP3	&	(2,2)	&	\cite{fliege2016sqp} (Table 1)	\\
6	&	CEC09\_7	&	(2,30)	&	\cite{fliege2016sqp} (Table 1)	&	37	&	MOP5	&	(3,2)	&	\cite{fliege2016sqp} (Table 1)	\\
7	&	CEC09\_8	&	(3,30)	&	\cite{fliege2016sqp} (Table 1)	&	38	&	MOP6	&	(2,4)	&	\cite{fliege2016sqp} (Table 1)	\\
8	&	CL1	&	(2,4)	&	\cite{fliege2016sqp} (Table 1)	&	39	&	P1	&	(2,5)	&	Appendix \ref{Test_Problems}	\\
9	&	Deb513	&	(2,2)	&	\cite{fliege2016sqp} (Table 1)	&	40	&	P2	&	(2,40)	&	Appendix \ref{Test_Problems}	\\
10	&	Deb521a\_a	&	(2,2)	&	\cite{fliege2016sqp} (Table 1)	&	41	&	P3	&	(2,7)	&	Appendix \ref{Test_Problems}	\\
11	&	Deb521b	&	(2,2)	&	\cite{fliege2016sqp} (Table 1)	&	42	&	P4a	&	(2,2)	&	Appendix \ref{Test_Problems}	\\
12	&	DTLZ1	&	(3,7)	&	\cite{fliege2016sqp} (Table 1)	&	43	&	P4b	&	(2,50)	&	Appendix \ref{Test_Problems}	\\
13	&	DTLZ1n2	&	(2,2)	&	\cite{fliege2016sqp} (Table 1)	&	44	&	P4c	&	(2,100)	&	Appendix \ref{Test_Problems}	\\
14	&	DTLZ2	&	(3,12)	&	\cite{fliege2016sqp} (Table 1)	&	45	&	P4d	&	(2,150)	&	Appendix \ref{Test_Problems}	\\
15	&	DTLZ2n2	&	(2,2)	&	\cite{fliege2016sqp} (Table 1)	&	46	&	P5a	&	(2,2)	&	Appendix \ref{Test_Problems}	\\
16	&	DTLZ3	&	(3,12)	&	\cite{fliege2016sqp} (Table 1)	&	47	&	P5b	&	(2,50)	&	Appendix \ref{Test_Problems}	\\
17	&	DTLZ3n2	&	(2,2)	&	\cite{fliege2016sqp} (Table 1)	&	48	&	P5c	&	(2,100)	&	Appendix \ref{Test_Problems}	\\
18	&	DTLZ5	&	(3,12)	&	\cite{fliege2016sqp} (Table 1)	&	49	&	P5d	&	(2,150)	&	Appendix \ref{Test_Problems}	\\
19	&	DTLZ5n2	&	(2,2)	&	\cite{fliege2016sqp} (Table 1)	&	50	&	P6a	&	(2,7)	&	Appendix \ref{Test_Problems}	\\
20	&	DTLZ6	&	(3,22)	&	\cite{fliege2016sqp} (Table 1)	&	51	&	P6b	&	(2,50)	&	Appendix \ref{Test_Problems}	\\
21	&	DTLZ6n2	&	(2,2)	&	\cite{fliege2016sqp} (Table 1)	&	52	&	P6c	&	(2,100)	&	Appendix \ref{Test_Problems}	\\
22	&	EP2	&	(2,2)	&	\cite{Evtushenko2013-ii}	&	53	&	P6d	&	(2,150)	&	Appendix \ref{Test_Problems}	\\
23	&	EX005	&	(2,2)	&	\cite{fliege2016sqp} (Table 1)	&	54	&	Shekel	&	(2,2)	&	\cite{Zilinskas2014-po}	\\
24	&	GE5	&	(3,3)	&	\cite{fliege2016sqp} (Table 1)	&	55	&	slcdt1	&	(2,2)	&	\cite{schutze2008convergence}	\\
25	&	IKK1	&	(3,2)	&	\cite{fliege2016sqp} (Table 1)	&	56	&	TP1	&	(2, 150)	&	\cite{gaspar2014evolutionary}	\\
26	&	Jin2\_a	&	(2,2)	&	\cite{fliege2016sqp} (Table 1)	&	57	&	TP3	&	(2, 100)	&	\cite{gaspar2014evolutionary}	\\
27	&	Jin3	&	(2,2)	&	\cite{fliege2016sqp} (Table 1)	&	58	&	TP4	&	(2, 100)	&	\cite{gaspar2014evolutionary}	\\
28	&	Jin4\_a	&	(2,2)	&	\cite{fliege2016sqp} (Table 1)	&	59	&	TP5	&	(2, 50)	&	\cite{gaspar2014evolutionary}	\\
29	&	lovison2	&	(2,2)	&	\cite{fliege2016sqp} (Table 1)	&	60	&	ZDT1	&	(2,30)	&	\cite{fliege2016sqp} (Table 1)	\\
30	&	lovison3	&	(2,2)	&	\cite{fliege2016sqp} (Table 1)	&	61	&	ZDT2	&	(2,30)	&	\cite{fliege2016sqp} (Table 1)	\\
31	&	lovison4	&	(2,2)	&	\cite{fliege2016sqp} (Table 1)	&	62	&	ZDT3	&	(2,30)	&	\cite{fliege2016sqp} (Table 1)	\\\hline

\end{tabular}
}
\caption{Details of the test problems.}
\label{table_1}
\end{table}

\paragraph*{Performance evaluation}

The proposed OMFFM (Algorithm~\ref{alg1}) is compared with MOSQCQP~\cite{ansary2021sqcqp} and the evolutionary algorithm NSGA-II~\cite{deb2002fast}. The comparison is conducted over the test suite listed in Table~\ref{table_1}. Performance is summarized using Dolan--Mor\'e performance profiles \cite{dolan2002benchmarking} based on five standard multi-objective indicators: purity, $\Delta$-spread, $\Gamma$-spread, hypervolume, and the total number of function evaluations. Further background on profile-based benchmarking and metric-driven comparisons can be found in \cite{ansary2020sequential, ansary2021sqcqp, fliege2016sqp, adhikary2024tunneling}. Figures~\ref{fig_purity}--\ref{fig_feval} report the resulting performance profiles.


\begin{figure}[H]
    \centering
    \begin{subfigure}[b]{0.45\textwidth}
        \centering
        \includegraphics[width=\textwidth,height=2.3cm]{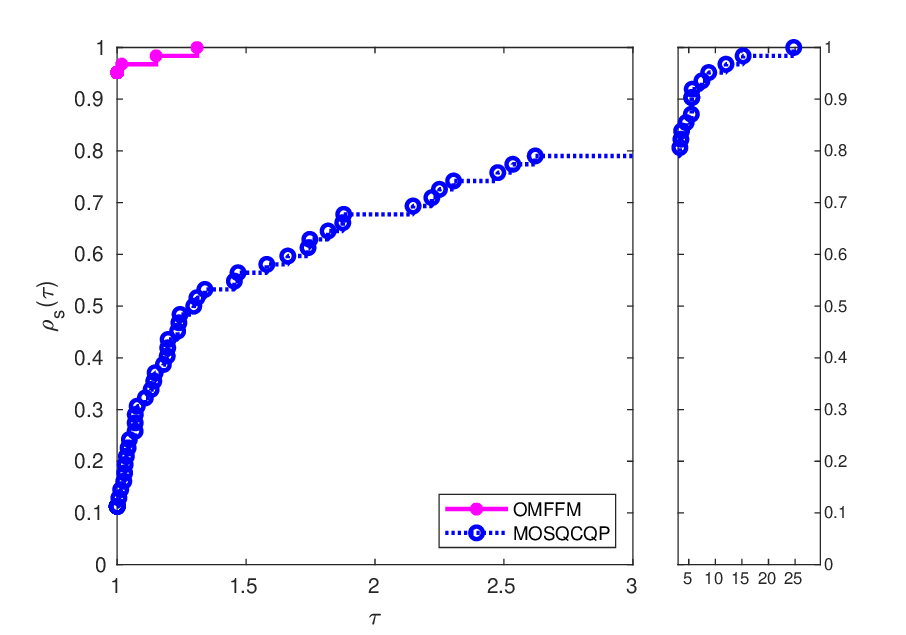}
        \caption{\raggedright Performance profile: OMFFM--MOSQCQP}
        \label{fig_purity_omffm_mosqcqp}
    \end{subfigure}
    \hfill
    \begin{subfigure}[b]{0.45\textwidth}
        \centering
        \includegraphics[width=\textwidth,height=2.3cm]{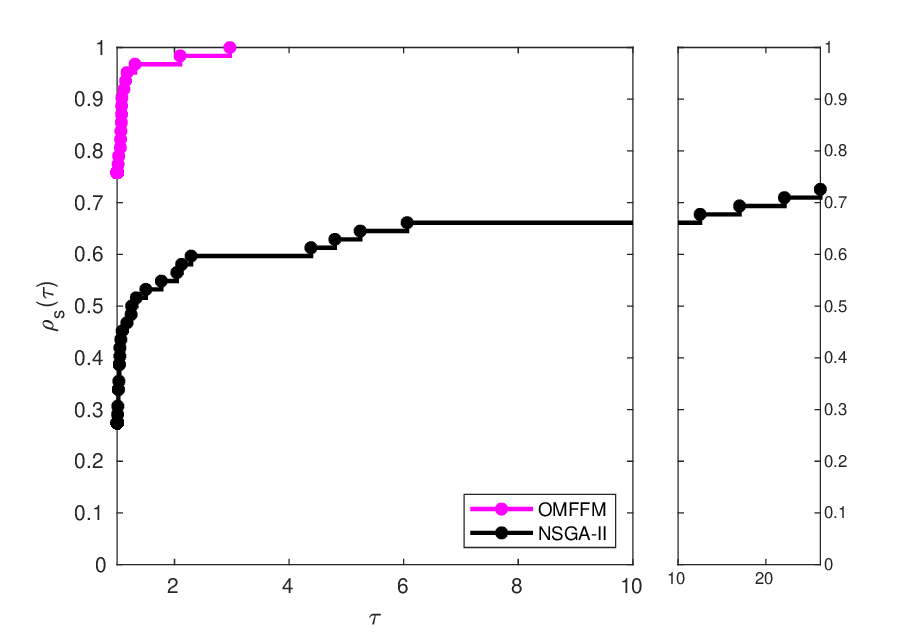}
        \caption{\raggedright Performance profile: OMFFM--NSGA-II}
        \label{fig_purity_omffm_nsga2}
    \end{subfigure}    
    \caption{Performance profiles based on purity metric}
    \label{fig_purity}
\end{figure}
\begin{figure}[H]
    \centering
    \begin{subfigure}[b]{0.45\textwidth}
        \centering
        \includegraphics[width=\textwidth,height=2.3cm]{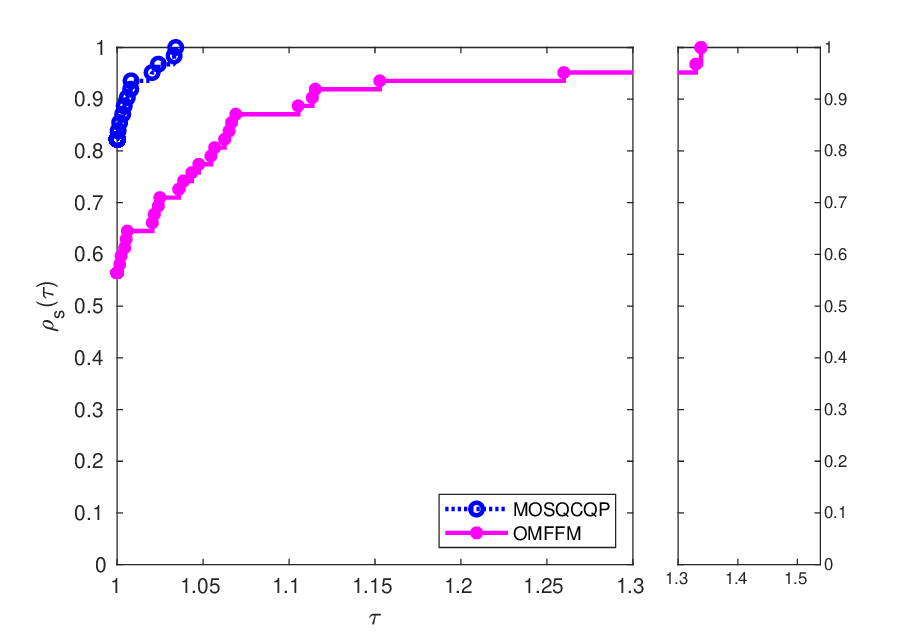}
        \caption{\raggedright Performance profile: OMFFM--MOSQCQP}
        \label{fig_delta_omffm_mosqcp}
    \end{subfigure}
    \hfill
    \begin{subfigure}[b]{0.45\textwidth}
        \centering
        \includegraphics[width=\textwidth,height=2.3cm]{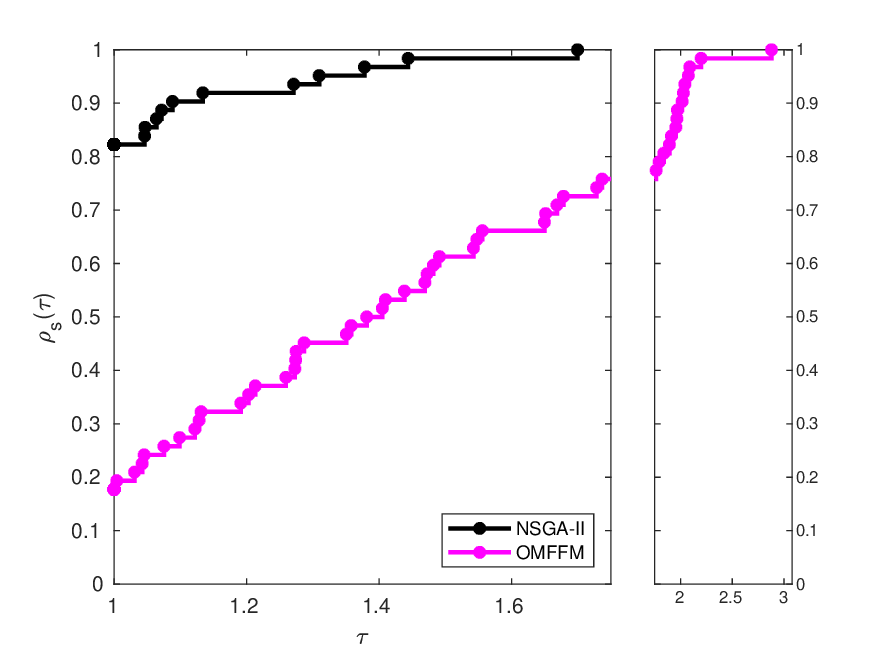}
        \caption{\raggedright Performance profile: OMFFM--NSGA-II}
        \label{fig_delta_omffm_nsga2}
    \end{subfigure}    
    \caption{Performance profiles based on $\Delta$ metric}
    \label{fig_delta}
\end{figure}
\begin{figure}[H]
    \centering
    \begin{subfigure}[b]{0.45\textwidth}
        \centering
        \includegraphics[width=\textwidth,height=2.3cm]{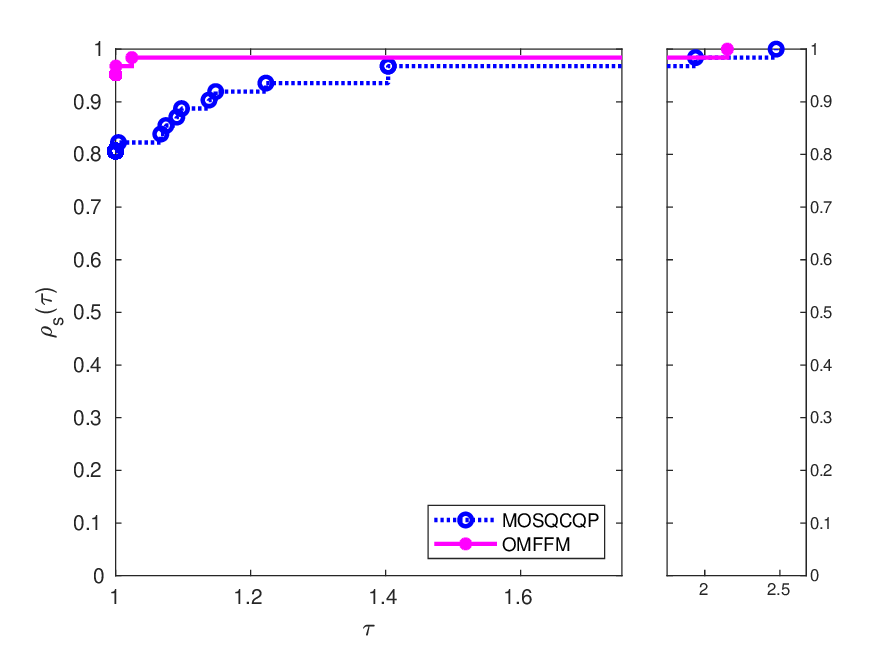}
        \caption{\raggedright Performance profile: OMFFM--MOSQCQP}
        \label{fig_gamma_omffm_mosqcqp}
    \end{subfigure}
    \hfill
    \begin{subfigure}[b]{0.45\textwidth}
        \centering
        \includegraphics[width=\textwidth,height=2.3cm]{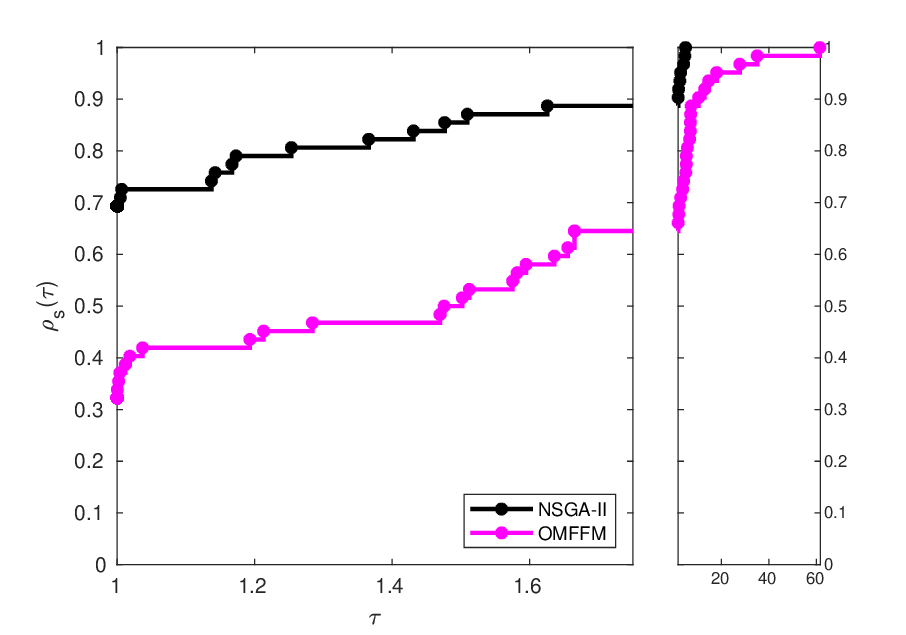}
        \caption{\raggedright Performance profile: OMFFM--NSGA-II}
        \label{fig_gamma_omffm_nsga2}
    \end{subfigure}    
    \caption{Performance profiles based on $\Gamma$ metric}
    \label{fig_gamma}
\end{figure}
\begin{figure}[H]
    \centering
    \begin{subfigure}[b]{0.45\textwidth}
        \centering
        \includegraphics[width=\textwidth,height=2.3cm]{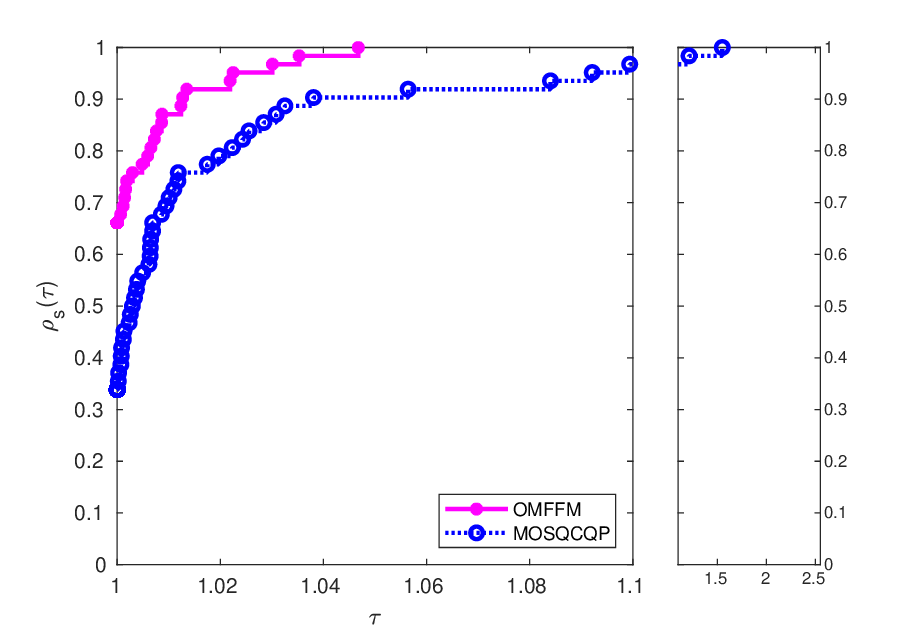}
        \caption{\raggedright Performance profile: OMFFM--MOSQCQP}
        \label{fig_hv_omffm_mosqcqp}
    \end{subfigure}
    \hfill
    \begin{subfigure}[b]{0.45\textwidth}
        \centering
        \includegraphics[width=\textwidth,height=2.3cm]{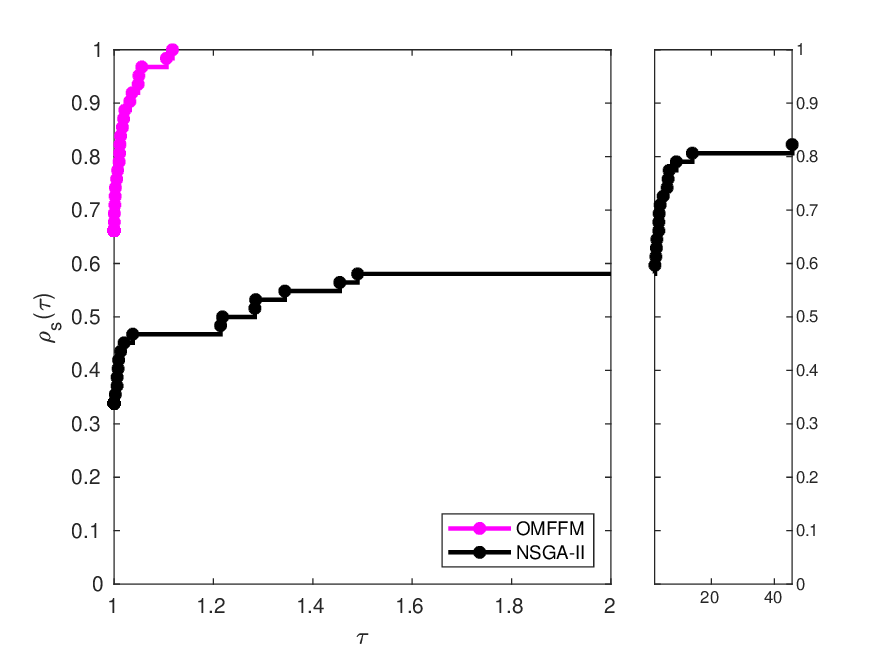}
        \caption{\raggedright Performance profile: OMFFM--NSGA-II}
        \label{fig_hv_omffm_nsga2}
    \end{subfigure}    
    \caption{Performance profiles based on hypervolume metric}
    \label{fig_hv}
\end{figure}
\begin{figure}[H]
    \centering
    \begin{subfigure}[b]{0.45\textwidth}
        \centering
        \includegraphics[width=\textwidth,height=2.3cm]{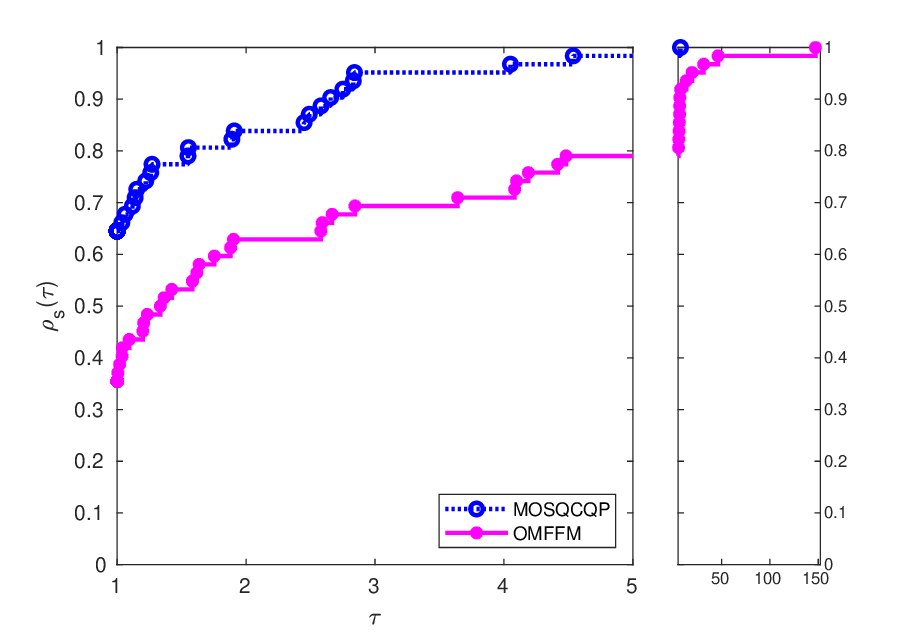}
        \caption{\raggedright Performance profile: OMFFM--MOSQCQP}
        \label{fig_feval_omffm_mosqcqp}
    \end{subfigure}
    \hfill
    \begin{subfigure}[b]{0.45\textwidth}
        \centering
        \includegraphics[width=\textwidth,height=2.3cm]{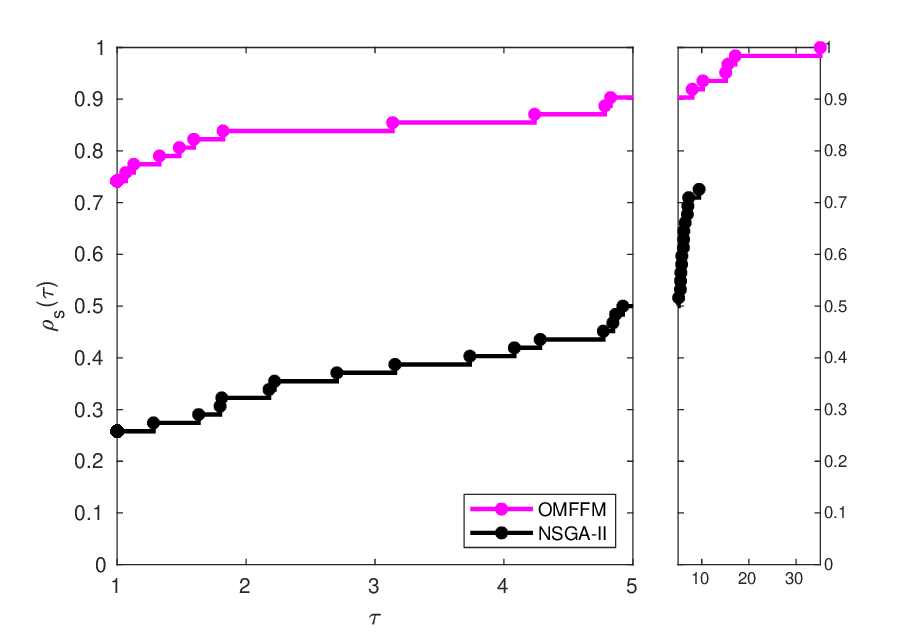}
        \caption{\raggedright Performance profile: OMFFM--NSGA-II}
        \label{fig_feval_omffm_nsga2}
    \end{subfigure}    
    \caption{Performance profiles based on the number of function evaluations}
    \label{fig_feval}
\end{figure}

Let $\mathcal{S}$ denote the set of methods under consideration and $\mathcal{P}$ the set of test problems.
For each $p\in\mathcal{P}$ and $s\in\mathcal{S}$, let $s_{p,s}$ denote the performance of method $s$ on problem $p$
with respect to a fixed indicator. The performance ratio is defined by $r_{p,s}=\frac{s_{p,s}}{\min_{s'\in\mathcal{S}} s_{p,s'}}$, so that the best method on problem $p$ satisfies $r_{p,s}=1$. The corresponding profile function is $$\rho_s(\tau)=\frac{\bigl|\{\,p\in\mathcal{P}:\ r_{p,s}\le \tau\,\}\bigr|}{|\mathcal{P}|}, \qquad \tau\ge 1.$$ The above equation gives the fraction of problems for which $s$ is within a factor $\tau$ of the best observed performance. Hence, higher curves indicate stronger overall performance.

The purity metric (Figure~\ref{fig_purity}) shows a clear advantage for OMFFM. The proposed method attains a large $\rho_s(1)$ and reaches $\rho_s(\tau)=1$ at relatively small $\tau$ in both comparisons, whereas MOSQCQP and NSGA-II require substantially larger ratios to achieve full coverage. For the $\Delta$-metric (Figure~\ref{fig_delta}), a spacing trade-off is observed: MOSQCQP and NSGA-II dominate for small $\tau$, while OMFFM approaches full coverage only for larger $\tau$, indicating
less uniform spacing on a subset of instances. The $\Gamma$-profiles in
Figure~\ref{fig_gamma} place OMFFM above MOSQCQP over most of the $\tau$-range and show earlier attainment of full coverage; however, NSGA-II exhibits stronger performance for small-to-moderate $\tau$, reflecting more consistent $\Gamma$-quality. For hypervolume (Figure~\ref{fig_hv}), OMFFM again achieves the highest profile over the informative
$\tau$-range and reaches full coverage sooner, reinforcing the convergence advantage suggested by purity. Finally, Figure~\ref{fig_feval} compares computational effort via
the number of function evaluations. Since MOSQCQP constitutes a core component of OMFFM's local phase, OMFFM is expected to require more evaluations than MOSQCQP; nonetheless, OMFFM clearly requires fewer function evaluations than NSGA-II.

Overall, the profiles indicate strong convergence behavior for OMFFM, as reflected by its dominance in purity and hypervolume, while also delivering improved $\Gamma$-performance relative to MOSQCQP. Although $\Delta$ suggests a trade-off in spacing uniformity, OMFFM
remains competitive and attains full coverage once modest performance tolerances are allowed. From a computational perspective, OMFFM retains practical efficiency by requiring fewer function evaluations than NSGA-II, and the higher evaluation budget relative to MOSQCQP is consistent with the added local-search component. Taken together, the evidence
across purity, $\Delta$, $\Gamma$, hypervolume, and function evaluations supports OMFFM as a robust and well-balanced method among the compared approaches for the considered multi-objective problems.

\section{Conclusion}\label{sec_con}
In this paper, a framework for multi-objective optimization is developed based on an auxiliary function, termed the multi-objective filled function. This function provides a mechanism to escape local weak efficient solutions and to generate globally efficient solutions in non-convex settings. The method requires no scalarization weights. It also avoids any kind of prior ordering or priority information among the objectives. Moreover, the filled function involves only one tuning parameter, which simplifies practical implementation. Using the established theoretical properties, an algorithm is derived to compute globally efficient solutions and to obtain an approximation of the Pareto front. The approach is assessed on widely used benchmark test problems and performance profiles are used to compare it with MOSQCQP and NSGA-II. The results indicate strong overall performance across the test suite, with particularly effective behavior on challenging non-convex problems. The analysis currently assumes smooth objectives. Future work will extend the approach to constrained and nonsmooth multi-objective optimization problems.



\section*{Funding}

This work was partially supported by the Prime Minister's Research Fellows (PMRF) Scheme, Government of India. Author Bikram Adhikary acknowledges support from the PMRF Scheme (PMRF ID: 2202747).


\section*{Statements and declarations}

\paragraph*{Conflict of interest}
The authors declare that they have no conflict of interest.

\paragraph*{Replication of results}
Implementation details of the proposed algorithm are explained in detail in Section \ref{sec_imp_exp}.  To facilitate reproducibility, the MATLAB implementation used to generate the results reported in this manuscript is available from the corresponding author upon reasonable request for non-commercial research use.



\paragraph*{Ethics approval and consent to participate}
Not applicable.


\paragraph*{Data availability}
No datasets were generated or analyzed in this paper.













\bibliographystyle{tfs}
\bibliography{Bikram_ref}

\appendix

\section{Details of test problems}\label{Test_Problems}
A number of the test problems in Table~\ref{table_1} are newly constructed in this work; they are given below.

\begin{enumerate}

\item P1 ($m=2$, $n=5$)
\begin{align*}
\min_{y \in \mathbb{R}^{5}} \quad & \left(\obj{1}(y),\,\obj{2}(y)\right) \\
\text{s.t.} \quad & -1 \le y_i \le 1,\quad \forall i \in \Idx{n}, \\
\text{where} \quad
& \obj{1}(y)= y_1^{2}-y_2^{2}+2y_3^{2}-3y_4^{2}, \\
& \obj{2}(y)= \sum_{i=1}^{5} y_i^{4}-\sum_{i=1}^{5} y_i^{2}.
\end{align*}

\item P2 ($m=2$, $n=40$)
\begin{align*}
\min_{y \in \mathbb{R}^{40}} \quad & \left(\obj{1}(y),\,\obj{2}(y)\right) \\
\text{s.t.} \quad & 0 \le y_i \le 1,\quad \forall i \in \Idx{n}, \\
\text{where} \quad 
& g(y)=\sum_{i=2}^{40}\left(y_i-\tfrac12\right)^{2}, \\
& \obj{1}(y)=\big(1+g(y)\big)\cos\!\left(\frac{\pi y_1}{2}\right), \\
& \obj{2}(y)=\big(1+g(y)\big)\sin\!\left(\frac{\pi y_1}{2}\right).
\end{align*}

\item P3 ($m=2$, $n=7$)
\begin{align*}
\min_{y \in \mathbb{R}^{n}} \quad & \left(\obj{1}(y),\,\obj{2}(y)\right) \\
\text{s.t.} \quad & 0 \le y_i \le 1,\quad \forall i \in \Idx{n}, \\
\text{where} \quad 
& g(y)=\sum_{i=2}^{n}\left(y_i-\tfrac12\right)^{2}, \\
& \obj{1}(y)=\big(1+g(y)\big)\, y_1\Big(1+0.2\sin\!\big(10\pi y_1\big)\Big), \\
& \obj{2}(y)=\big(1+g(y)\big)\, (1-y_1)\Big(1+0.2\cos\!\big(10\pi y_1\big)\Big).
\end{align*}

\item P4 ($m=2$, $n \in \{2,50,100,150\}$)
\begin{align*}
\min_{y \in \mathbb{R}^{n}} \quad & \left(\obj{1}(y),\,\obj{2}(y)\right) \\
\text{s.t.} \quad & 0 \le y_i \le 1,\quad \forall i \in \Idx{n}, \\
\text{where} \quad
& \obj{1}(y)= y_1, \\
& \obj{2}(y)= g(y)\left(1-\left(\frac{y_1}{g(y)}\right)^2-0.3\,y_1\sin(10\pi y_1)\right), \\
& g(y)=1+\frac{4}{n-1}\sum_{i=2}^{n} y_i^2.
\end{align*}

\item P5 ($m=2$, $n \in \{2,50,100,150\}$)
\begin{align*}
\min_{y \in \mathbb{R}^{n}} \quad & \left(\obj{1}(y),\,\obj{2}(y)\right) \\
\text{s.t.} \quad & 0 \le y_i \le 1,\quad \forall i \in \Idx{n}, \\
\text{where} \quad
& \obj{1}(y)= y_1, \\
& \obj{2}(y)= g(y)\left(1-\frac{y_1}{g(y)}-0.4\,e^{-5y_1}\sin(8\pi y_1)\right), \\
& g(y)=1+\frac{3}{n-1}\sum_{i=2}^{n} y_i.
\end{align*}

\item P6 ($m=3$, $n \in \{7,50,100,150\}$)
\begin{align*}
\min_{z \in \mathbb{R}^n} \quad & \left( \obj{1}(z),\; \obj{2}(z),\; \obj{3}(z) \right) \\
\text{s.t.} \quad & 0 \le z_i \le 1, \quad \forall i \in \Idx{n}, \\
\text{where} \quad
& g(z) = 1 + 10 \sum_{j=1}^{\lfloor n/2 \rfloor} z_{2j}
          + 5 \sum_{j=1}^{\lfloor (n-1)/2 \rfloor} z_{2j+1}, \\
& r(z) = \frac{z_1}{g(z)}, \\
& \obj{1}(z) = z_1, \\
& \obj{2}(z) = g(z)\left( 1 - r(z)^2 - r(z)\sin(8\pi z_1) \right), \\
& \obj{3}(z) = g(z)\left( 1 - r(z)^2 - r(z)\cos(12\pi z_1) \right).
\end{align*}
\end{enumerate}

\end{document}